\documentclass{amsart}
\usepackage{amssymb}
\usepackage[dvips]{epsfig}
\usepackage{graphicx}
\usepackage{color}

 \usepackage[latin1]{inputenc}
 \usepackage[T1]{fontenc}
 \usepackage[normalem]{ulem}
\usepackage{verbatim}
 \usepackage{graphicx}
\usepackage[latin1]{inputenc}
\usepackage{amsmath}
\usepackage{amssymb}
\usepackage{amsfonts}
\usepackage{amsthm}
\usepackage{bbm}

\newcommand{\R}{\mathbb R}
\newcommand{\p}{\partial}

\newcommand{\Z}{\mathbb Z}

\newcommand{\C}{\mathbb C}

\renewcommand{\a}{\alpha}

\newcommand{\tres}{|\!|\!|}

\newcommand{\ji}{\langle}
\newcommand{\jd}{\rangle}
\newtheorem{theorem}{Theorem}[section]
\newtheorem{lemma}[theorem]{Lemma}

\theoremstyle{remark}
\newtheorem{remark}{Remark}[section]
\theoremstyle{definition}

\numberwithin{equation}{section}

\begin{document}
\title[Generalized derivative Schr\"odinger equation]{On a class of solutions to the generalized derivative Schr\"odinger equations}
\author{F.  Linares}
\address[F. Linares] {IMPA\\ Estrada Dona Castorina 110, Rio de Janeiro 22460-320, RJ Brazil}
\email{linares@impa.br}
\author{G. Ponce}
\address[G. Ponce]{Department  of Mathematics\\
University of California\\
Santa Barbara, CA 93106\\
USA}
\email{ponce@math.ucsb.edu}
\author{G.N. Santos}
\address[G.N. Santos]{Universidade Federal do Piau\'i - UFPI.
Campus Ministro Petr\^onio Portella\\
Teresina 64049-550, PI Brazil}
\email{gleison@ufpi.edu.br}

\keywords{}

\begin{abstract}

In this work we shall consider the initial value problem associated to the generalized derivative Schr\"odinger equations
\begin{equation*}
\p_tu=i\p_x^2u + \mu\,|u|^{\a}\p_xu, \hskip10pt x,t\in\R,  \hskip5pt 0<\a \le 1\;\; {\rm and}\;\; |\mu|=1,
\end{equation*}
and
\begin{equation*}
\p_tu=i\p_x^2u + \mu\,\p_x\big(|u|^{\a}u\big), \hskip10pt x,t\in\R,  \hskip5pt 0<\a \le 1\;\; {\rm and}\;\; |\mu|=1.
\end{equation*}
Following the argument introduced by Cazenave and Naumkin \cite{Cazenave} we shall establish the local well-posedness for
a class of small data in an appropriate weighted Sobolev space. The other main tools in the proof include the homogeneous
and inhomogeneous versions of the Kato smoothing effect for the linear Schr\"odinger equation established by Kenig-Ponce-Vega in \cite{KPV1}.
\end{abstract}

\maketitle


\large

\section{introduction}

We study the initial value problems (IVP) associated to the generalized nonlinear derivative Schr\"odinger equation,
\begin{equation}\label{gdnls}
\begin{cases}
\p_tu=i\,\p_x^2u + \mu\,|u|^{\a}\p_xu, \hskip15pt x,t\in\R, \hskip5pt 0<\a \le 1,\\
u(x,0)=u_0(x),
\end{cases}
\end{equation}
and
\begin{equation}\label{gdnls-b}
\begin{cases}
\p_tu=i\,\p_x^2u + \mu\,\p_x\big(|u|^{\a}u\big), \hskip15pt x,t\in\R, \hskip5pt 0<\a \le 1,\\
u(x,0)=u_0(x),
\end{cases}
\end{equation}
where $u$ is a complex valued function, $\mu\in \C$ with $|\mu|=1$.

The equation in \eqref{gdnls-b} generalized the well-known derivative nonlinear Schr\"odinger equation (DLNS),
\begin{equation}\label{dnls-0}
i\p_tu+\p_x^2u + i\p_x(|u|^2 u)=0, \hskip15pt x,t\in\R,
\end{equation}
which appears as a model in plasma physics and optics (\cite{Mio}, \cite{Mijolhus}, \cite{MMW}).  It is also an equation which is  exactly solvable by the inverse scattering technique, see \cite{KN}. The gauge transformation
\begin{equation}\label{dnls-1}
v(x,t)= u(x,t)\,\exp\Big(\frac{i}{2} \int\limits_{-\infty}^x |u(y,t)|^2\,dy\Big)
\end{equation}
allows to obtain the Hamiltonian form of the equation \eqref{dnls-0}, i.e.
\begin{equation}\label{dnls}
i\p_t v+\p_x^2v+i|v|^2\p_xv=0, \hskip10pt x, t\in\R.
\end{equation}
More precisely,
\begin{equation}\label{dnls-2}
\frac{dv(t)}{dt}= -i E'(v(t))
\end{equation}
where $E(v)$ is the energy of $v$ defined by
\begin{equation}\label{dnls-3}
E(v)(t)=\frac12 \int |\p_x v|^2\,dx+\frac14\, {\rm Im} \int |v|^2\bar{v}\p_xv\,dx.
\end{equation}

The IVP associated to \eqref{dnls} has been studied
in several publications  (see for instance \cite{Linares}, \cite{Iteam}, \cite{refinement}, \cite{HayashiII}, \cite{HayashiOzawa}, \cite{HayashiOzawa2}, \cite{Ozawa}, \cite{Takaoka},  \cite{Takaoka2}) where among other qualitative properties local and global well-posedness issues were investigated. In particular,  a global sharp well-posedness result was obtained by  Guo and Wu \cite{GW} in $H^{1/2}(\R)$ for initial data satisfying 
\begin{equation}
\label{12}
\|u_0\|_2^2\le {4\pi},
\end{equation}
see also \cite{W}.  The condition \eqref{12} guarantees the energy $E(\cdot)$ is positive via sharp Gagliardo-Nirenberg inequality.

 Recently, via inverse scattering method Jenkins, Liu, Perry and Sulem \cite{JLPS} established global existence of solutions without any restriction on the size of the data in an appropriate weighted Sobolev space.  For results concerning the initial-periodic-boundary value problem (IPBVP)  we refer to \cite{Grunrock}, and \cite{Herr}.

Like the DNLS equation the gDNLS admits a two-parameter family of solitary wave solutions given explicitly by 
\begin{equation*}
\psi_{\omega,c}(x,t)=\varphi_{\omega,c}(x-ct) \exp i\Big\{ wt+\frac{c}{2}(x-ct)-\frac{1}{\a+2} \int\limits_{-\infty}^{x-ct} \varphi_{\omega,c}^{\a}(y)\,dy\Big\}, \hskip10pt \a>0,
\end{equation*}
where
\begin{equation*}
\varphi_{w,c}(x)=
\begin{cases}
\Big\{\dfrac{(2+\a)(4\omega-c^2)}{4\sqrt{\omega}\cosh(\frac{\a}{2}\sqrt{4\omega-c^2}x\big)-2c}\Big\}^{\frac{1}{\a}},  \hskip20pt \text{if} \hskip10pt \omega>\frac{c^2}{4},\\
\\
(\a+2)^{\frac{1}{\a}} \,c^{\frac{1}{\a}}\,(\frac{\a^2}{4}(cx)^2+1)^{-\frac{1}{\a}}, \hskip7pt \text{if} \hskip10pt \omega=\frac{c^2}{4} \hskip10pt \text{and} \hskip10pt  c>0.
\end{cases}
\end{equation*}
In \cite{Sulem}  Liu, Simpson and Sulem obtained stability  and instability results for these solitary wave solutions (see also \cite{TX} and references therein).
Concerning stability matters, for the case $\a=2$   the set of available results is significantly more complete (see \cite{Guo}, \cite{Colin}, \cite{KW}, \cite{MTX}, \cite{LW} and references
therein).
 
When $\a>0$, $\a\ne2$, very little is known regarding well-possedness results for the IVP \eqref{gdnls}.  In \cite{Hao} Hao proved 
 local well-posedness  in $H^s(\R)$, for $\a>5$ and $s>\frac12$. Ambrose and Simpson \cite{Ambrose} for the IPBVP established existence and uniqueness of solutions $u\in L^{\infty}([0,T]; H^1(\mathbb T))$. In \cite{Gleison} Santos showed
the existence and uniqueness of solution $u\in L^{\infty}((0,T) ; H^{3/2}(\R)\cap \ji x\jd^{-1}H^{1/2}(\R))$ for
sufficient small initial data in the case  $1<\a<2$, and local well-posedness in $H^{1/2}(\R)$ for small data when $\a>2$. In \cite{ MHO}  Hayashi and Ozawa
considered the gDNLS equation  in a bounded interval with a Dirichlet condition and established local results in $H^2$ for $\a\ge1$ and $H^1$ for $\a\ge 2$. In \cite{FHI}   Fukaya, Hayashi and Inui showed global result for initial data in $H^1(\R)$, for any $\a\ge 2$, with initial data satisfying $\|u_0\|_2^2\le {4\pi}$.
Their approach is based in a variational method. No result for the case $\alpha\in (0,1]$ has been available. 

Our goal in this work is to give a positive answer to the local well-posedness for the IVP \eqref{gdnls} in a class of initial data.  To present our result we first describe our motivation and the ideas behind the proofs.

In \cite{Cazenave}  Cazenave and Naumkin studied the IVP associated to semi-linear Schr\"odinger equation,
\begin{equation}\label{nls}
\begin{cases}
\p_tu=i(\Delta u\pm|u|^{\a} u), \hskip15pt x\in\R^N,\;t\in\R, \hskip5pt \a>0,\\
u(x,0)=u_0(x).
\end{cases}
\end{equation}
with initial data in $u_0\in H^s(\R^N)$.  For every $\a>0$ they constructed a class of initial data for which there exist  unique local solutions for the IVP \eqref{nls}.  Also, they obtained 
a class of initial data for $\alpha>\frac{2}{N}$ for which there exist global solutions that scatter.

One of the ingredients in the proofs
of their results, is the fact that solutions of the linear problem satisfy
\begin{equation}\label{lower}
\underset{x\in\R^N}{\rm Inf}\; \ji x\jd^m\,| e^{it\Delta} u_0(x)| >0,
\end{equation}
for $t\in[0,T]$ with $T$ sufficiently small whenever the initial data satisfy
\begin{equation}\label{lower-2}
\underset{x\in\R^N}{\rm Inf}\; \ji x\jd^m\,|u_0(x)|\ge \lambda>0.
\end{equation}
This is reached for $m=m(\a)$ and $u_0\in H^s(\R^N)$ with $s$ sufficiently large. To prove the inequality \eqref{lower} the authors in \cite{Cazenave}  
rely on  Taylor's power expansion to avoid applying the Sobolev embedding
since the nonlinear $|u|^{\alpha} u$ is not regular enough and it would restrict the argument to dimensions $N\ge 4$. 

In the case on study in this paper, besides the term $|u|^{\alpha}$ being not regular enough  (it is only  $C^{\alpha}$ for $\;0<\alpha<1$), the presence of the derivative makes more difficult the analysis of the nonlinear term of the equation in \eqref{gdnls}.  

Inspired in the results in \cite{Cazenave} and using the smoothing effects of Kato type for the linear Schr\"odinger equation (homogeneous and inhomogeneous versions found
in \cite{KPV1}) we 
establish a local well-posedness result for the IVP \eqref{gdnls} for small data in a Sobolev weighted space with data
satisfying \eqref{lower-2}.  We achieve this via a contraction mapping principle applied to the integral equation equivalent form for the IVP \eqref{gdnls},
\begin{equation}\label{lower-3}
u(t)=e^{it\p_x^2}u_0+\mu\, \int\limits_0^t e^{i(t-t')\p_x^2} (|u|^{\a}\p_xu)(t')\,dt'.
\end{equation}
and 
\begin{equation}\label{lower-3b}
u(t)=e^{it\p_x^2}u_0 +\mu\, \int\limits_0^t e^{i(t-t')\p_x^2} \p_x(|u|^{\a}u)(t')\,dt'.
\end{equation}
for the IVP \eqref{gdnls-b}.

More precisely,  our main result reads as follows.



\begin{theorem}\label{main}
Given $\a\in (0, 1)$ and $k\ge   m+3$, \hskip3pt $k\in\Z^{+}$, with $\,m\equiv\Big[\frac{2}{\alpha}+1\Big]\, $ 
{\rm(}$\,[x]$  the greatest integer less than or equal to $\,x${\rm)}, 
there exists  $\delta=\delta(\alpha;k)>0$ such that if
\begin{equation}
\label{sob}
u_0\in H^s(\R), \hskip15pt s=k+\frac12,
\end{equation}
and
\begin{equation}
\label{newspace}
\ji x\jd^m u_0\in  L^{\infty}(\R),\hskip15pt  \ji x\jd^m \partial_x^{j+1}u_0\in  L^{2}(\R),\;\;\;\;\;\;\;j=0,1,2,
\end{equation}

with

\begin{equation}\label{main-1}
\|u_0\|_{s,2}+\| \ji x\jd^m u_0\|_{\infty}+\sum_{j=1}^3\| \ji x\jd^m \partial_x^ju_0\|_2 < \delta
\end{equation}
and
\begin{equation}\label{main-2}
\underset{x}{\rm Inf}\; \ji x\jd^m\,|u_0(x)|\ge \lambda>0.
\end{equation}
Then the IVP  \eqref{gdnls} has a  unique solution $u$ such that
\begin{equation}\label{main-3}
u\in C([0,T] : H^s(\R)),
\end{equation}
\begin{equation}\label{main-3a}
\ji x\jd^m u\in C([0,T] : C(\R)\cap L^{\infty}(\R)),
\end{equation}
\begin{equation}\label{main-3b}
\ji x\jd^m\partial_x^{j+1}u\in C([0,T] :L^2(\R)),\;\;\;\;\;\;\;\;j=0,1,2,
\end{equation}
and
\begin{equation}\label{main-4}
\p_x^{k+1}u\in L^{\infty}(\R : L^2([0,T]))
\end{equation}
with $T=T(\alpha; k; \delta; \lambda)>0$. Moreover, the map data-solution
\begin{equation*}
u_0\mapsto u(\cdot,t)
\end{equation*}
from a neighborhood of the datum $u_0$ in  $H^s(\R)$  intersected with the set in \eqref{newspace}  satisfying \eqref{main-1}-\eqref{main-2} into the class defined by 
\eqref{main-3}-\eqref{main-4}
is smooth.
\end{theorem}

\vskip.25in
\begin{remark}

In the case $0<\a<2$, $\mu\neq\pm 1$, the weighted condition on the data can be explained by the so called Mizohata condition (simplified version) \cite{Miz}.

For the linear IVP
\begin{equation}\label{mizohata}
\begin{cases}
\p_t u=i\Delta u+\vec{b}(x)\cdot \nabla u,\hskip10pt x\in\R^n, \;t\in\R,\\
u(x,0)=u_0(x),
\end{cases}
\end{equation}
where $\vec{b}= (b_1, \dots, b_n)$ with $b_j:\R \to \C$, $j=1,\dots,n$ smooth functions, the hypothesis 
\begin{equation}\label{MC}
\underset{\hat{\xi}\in \mathbb{S}^{n-1}}{\rm Sup}\; \underset{l\in \R}{\underset{x\in\R^n}{\rm Sup}}\;\Big|\int\limits_0^l {\rm Im} \, b_j(x+r\hat{\xi}\,)\,\hat{\xi}_j\,dr\Big| <\infty
\end{equation}
is a necessary condition for the $L^2$-well-posedness of \eqref{mizohata}.
\end{remark}

\begin{remark}

(a) The choice $\,m=[2/\a+1]\,$ can be replaced by $m>1/\a\geq 1$. Notice that this hypothesis is consistent with Mizohata's condition \eqref{MC}. However, by fixing  $m=[2/\a+1]$ one greatly shortens  some technical detail in the proof. \\
Here the constants $\delta=\delta(s;\a)>0$ and $m=m(\a)$, and clearly $\lambda\ll \delta$. Typically the function 
$\varphi(x)=\dfrac{c_0}{\ji x\jd^m}$ for appropriate $c_0$ satisfies the hypotheses.
\vskip.15in
(b)  The condition $k>m$ is natural since the operator $\Gamma= x+ 2it\p_x$ commutes with $\p_t-i\p_x^2$. 
\vskip.15in

(c) The hypothesis on $s$, $s-1/2=k\in\Z^{+}$ is not essential, but highly simplify the exposition around the use of the smoothing effect in \cite{KPV1}, which in the inhomogenous case roughly speaking provides a gain of one derivative, see Lemma \ref{lemma2.1}, estimate \eqref{gdnls-4}.

\vskip.15in

(d) From the assumptions \eqref{newspace} and \eqref{main-2} one has that
$$
\ji x\jd^{m-1}u_0\in  L^2(\R),\;\;\;\;\;\;\;\ji x\jd^{m}u_0\notin  L^2(\R)
$$
and by Sobolev embedding
$$
\ji x\jd^{m}\partial_xu_0,\;\;\ji x\jd^{m}\partial_x^2u_0\in  L^{\infty}(\R).
$$

Also, fixing $ k=m+3$ from the interpolation in Section 2 it follows from \eqref{main-1} that
$$
\| \ji x\jd^{m-r}\partial_x^{1+l+r}u_0\|_2 ,\;\;\;\;l=0,1,\;\;\;\;\;r=0,..,m,
$$
can be interpolated between
$$
\| \ji x\jd^{m}\partial^{1+l}_xu_0\|_2\,\;\;\;\;\text{and}\;\;\;\;\;\|\partial_x^{1+l}u_0\|_{m,2}\;\;\;\;\;\;l=0,1.
$$
\end{remark}
\vskip.15in

Concerning the IVP \eqref{gdnls-b} we obtain the following:

\begin{theorem}\label{thm2}  Under the same hypotheses, the conclusions of Theorem \ref{main} extend to solutions of the IVP \eqref{gdnls-b}.
\end{theorem}

This paper is organized as follows, in Section 2, we list some estimates useful in the proof of Theorem \ref{main}. Section 3 contains the proof of our main result. 

Before leaving this section we will introduce the notation used in this manuscript.

\subsection{Notation} We denote $\ji x\jd= (1+x^2)^{1/2}$. The Fourier transform of a function $f$, and its inverse Fourier transform are denoted by 
$\hat{f}$ and $\check{f}$ respectively. For $s\in \R$, $ J^s=(1-\p_x^2)^{s/2}$, and $D^s=(-\p_x^2)^{s/2}$ stand for the Riesz and Bessel potentials of
order $-s$, respectively. The functional space $H^s(\R)=(1-\p_x^2)^{-s/2}(L^2(\R))$ denotes the Sobolev spaces of order $s$ endowed with the norm $\|f\|_{s,2}=\|J^s f\|_2$.
$H^l(\R:\,d\nu),\;l\in\Z^{+},$ denotes the weighted Sobolev space of functions $f :\R\to \R$ such that
$$
\|f\|^2_{H^l(\R:\,d\nu)}=\sum_{j=0}^l\,\int_{\R}((\partial_x^jf)(x))^2\,d\nu(x)<\infty.
$$

 For two variable functions $f=f(x,t)$ with $(x,t)\in \R\times [0,T]$, we consider the mixed spaces $L^q([0,T]; L^p(\R))$ and $L^p(\R; L^q([0,T]))$  
corresponding to the norms
\begin{equation*}
\|f\|_{L^q_TL^p_x}=\Big(\int\limits_0^T\Big(\int\limits_{\R} |f(x,t)|\,dx\Big)^{q/p}\,dt\Big)^{1/q}
\end{equation*}
and
\begin{equation*}
\|f\|_{L^p_xL^q_T}=\Big(\int\limits_{\R}\Big(\int\limits_{0}^T |f(x,t)|\,dt\Big)^{p/q}\,dx\Big)^{1/p}.
\end{equation*}

\section{preliminary estimates}

We shall use the following estimates concerning the inhomogeneous linear problem associated to the Schr\"odinger equation: 

\begin{lemma}
\label{lemma2.1} \cite{KPV1}  If $v$ is solution of 
\begin{equation}\label{gdnls-3}
\begin{cases}
\p_t v=i\p_x^2v + F(x,t), \hskip15pt x\in\R, \hskip5pt t\in [0,T],\\
v(x,0)=v_0(x),
\end{cases}
\end{equation}
then
\begin{equation}\label{gdnls-4}
\|D^{1/2}_x v(t)\|_{L^{\infty}_TL^2_x}+\|\p_xv(t)\|_{L^{\infty}_xL^2_T}\lesssim \|D^{1/2}_xv_0\|_2+\|F\|_{L^1_xL^2_T}.
\end{equation}
\end{lemma}
\begin{proof} For the proof of \eqref{gdnls-4}  we refer to \cite{KPV1}. 
\end{proof}

The result in Lemma \ref{lemma2.1}  is the optimal one-dimensional (homogeneous and inhomogeneous) version of the smoothing effect in solutions of the Schr\"odinger equation. This type of effect, in his homogenous version, was first established by T. Kato \cite{Ka} in solutions of the Korteweg-de Vries equation 
(see also \cite{KF}). In \cite{CS}, \cite{Sj} and \cite{Ve} the corresponding homogeneous $n$-dimensional result was deduced (for further comments and discussion see \cite{LP}).

To end this section,  we state the following useful interpolation results. 
\begin{lemma}\label{interp} For any $a, b>0$ and $\gamma\in (0,1)$ 
\begin{equation}
\begin{aligned}
&\|J^{\gamma a}(\ji x\jd^{(1-\gamma)b}f)\|_2 \le  c\,\|\ji x\jd^{b} f\|_2^{1-\gamma}\|J^a f\|^{\gamma}_2, \hskip10pt \gamma\in(0,1),\\
\text{and}&\\
&\|\ji x\jd^{\gamma a}(J^{(1-\gamma)b}f)\|_2 \le  c\,\|J^{b} f\|_2^{1-\gamma}\|\ji x\jd^a f\|^{\gamma}_2, \hskip10pt \gamma\in(0,1).
\end{aligned}
\end{equation} 
\end{lemma}

For the proof which is based on the Three Lines Theorem we refer to \cite{NP}.

By integration by parts, one also has :
\begin{lemma}\label{interp2} For any $j, \,k\in\Z^{+},\,j,\,k\geq 1$ there exists $c=c(k;j)>0$ such that
\begin{equation}
\begin{aligned}
& \| \ji x\jd^k\partial_x^jf\|_2^2\leq c \| \ji x\jd^k\partial_x^{j+1}f\|_2 \| \ji x\jd^k\partial_x^{j-1}f\|_2 +c \| \ji x\jd^{k-1}\partial_x^{j-1}f\|_2^2,\\
& \| \ji x\jd^k\partial_x^jf\|_2^2\leq c \| \ji x\jd^{k-1}\partial_x^{j+1}f\|_2 \| \ji x\jd^{k+1}\partial_x^{j-1}f\|_2 +c \| \ji x\jd^{k-1}\partial_x^{j-1}f\|_2^2,\\
& \| \ji x\jd^k\partial_x^jf\|_2^2\leq c \| \ji x\jd^{k+1}\partial_x^{j+1}f\|_2 \| \ji x\jd^{k-1}\partial_x^{j-1}f\|_2 +c \| \ji x\jd^{k-1}\partial_x^{j-1}f\|_2^2.
\end{aligned}
\end{equation} 
\end{lemma}
\vskip.15in

\vskip.2in

\section{Proof of the Main Result}
Consider the IVP
\begin{equation}\label{gdnls-1}
\begin{cases}
\p_tu=i\p_x^2u+\mu\,|u|^{\a}\p_xu, \hskip15pt x,t\in\R, \hskip5pt 0<\a \le 1,\;\;\;\;|\mu|=1,\\
u(x,0)=u_0(x)
\end{cases}
\end{equation}
and its integral equation version
\begin{equation}\label{gdnls-2}
u(t)=e^{it\p_x^2}u_0+\mu \int\limits_0^t e^{i(t-t')\p_x^2} (|u|^{\a}\p_xu)(t')\,dt'.
\end{equation}

Let us fix  $s=k+\frac12$, with $k$ to be determined below.

We shall establish the existence and uniqueness (in fact local well-posedness of \eqref{gdnls-1}) by using the contraction principle of the
operator
\begin{equation}\label{gdnls-5}
\Phi(u)(t)=e^{it\p_x^2}u_0+\mu\, \int\limits_0^t e^{i(t-t')\p_x^2} (|u|^{\a}\p_xu)(t')\,dt',
\end{equation}
with the data  $u_0$ satisfying :
\begin{equation}\label{gdnls-6}
\|u_0\|_{s,2}+ \|\ji x\jd^m u_0\|_{\infty}+\sum_{j=0}^2\| \ji x\jd^m \partial_x^{j+1}u_0\|_2 < \delta
\end{equation}
 for some $\delta\sim 1$ and for some $m+3\le k=s-1/2$, $k\ge 4$, there exists $\lambda>0$ such that
\begin{equation}\label{gdnls-7}
\underset{x}{\rm Inf}\, \ji x\jd^m |u_0(x)|\ge \lambda.
\end{equation}

It will be shown that if
\begin{equation}\label{gdnls-8}
\begin{split}
X_T&=\Big\{ u\in C([0,T]: H^s(\R)) : \\
&\underset{[0,T]}{\sup}\;\big(\|u(t)\|_{s,2}+\| \ji x\jd^m u(t)\|_{\infty}+\sum_{j=1}^3\| \ji x\jd^m \partial_x^{j}u(t)\|_2\big) \\
\\
&\;\;\;+\|\p_x^{k+1}u\|_{L^{\infty}_xL^2_T}\equiv \tres u \tres_{_T}\le 2c\delta\\
\\
\;\;\;\;\;&\text{and}\\
&\;\;\underset{(x,t)\in\R\times[0,T]}{\rm Inf}\, \ji x\jd^m |u(x,t)|\ge \frac{\lambda}{4}\;\;\Big\},
\end{split}
\end{equation}
then $\Phi$ maps $X_T$ into itself and defines a contraction in the $\,\tres\cdot\tres_{_T}$-norm for $T$ sufficiently small.

We recall that
\begin{equation}\label{gdnls-9}
\|f\|_{s,2} \simeq \|D^s_x f\|_2+\|f\|_2
\end{equation}

Thus, from \eqref{gdnls-4}
\begin{equation}\label{gdnls-10}
\begin{split}
\|D^s\Phi(u)&\|_{L^{\infty}_TL^2_x}+\|\p_x^{k+1}\Phi(u)\|_{L^{\infty}_xL^2_T}\\
& \le c\,\delta
+\underset{j=0}{\overset{k}{\sum}} c_j\,\|\p_x^j(|u|^{\a}) \p_x^{k+1-j}u\|_{L^1_xL^2_T}\\
&\le c\delta +\underset{j=0}{\overset{k}{\sum}} c_j\,A_j.
\end{split}
\end{equation}

We observe that from \eqref{gdnls-8} one has that
\begin{equation}\label{gdnls-11}
\ji x\jd^m\ge \frac{\lambda}{4}\, |u(x,t)|^{-1} \hskip 15pt \forall (x,t)\in \R\times [0,T].
\end{equation}

First consider the term $A_0$ in \eqref{gdnls-10}. Then for $\theta>1$
\begin{equation}\label{gdnls-12}
\begin{split}
A_0 &\equiv c_0\,\| \ji x\jd^{-\theta} \ji x\jd^{\theta} |u|^{\a}\p_x^{k+1}u\|_{L^1_xL^2_T}\\
&\le c\, \|\ji x\jd^{\theta} |u|^{\a}\p_x^{k+1}u\|_{L^{\infty}_xL^2_T}\\
&\le c\, \|\ji x\jd^{\theta} |u|^{\a}\|_{L^{\infty}_xL^{\infty}_T}\|\p_x^{k+1}u\|_{L^{\infty}_xL^2_T}\\
&\le c\, \|\ji x\jd^m |u|\|^{\a}_{L^{\infty}_xL^{\infty}_T}\|\p_x^{k+1}u\|_{L^{\infty}_xL^2_T}
\end{split}
\end{equation}
with $m\geq \theta/\a$.

To simplify the exposition we fix 
$$
\theta=2,\;\;\;\;\;\;\;m=\Big[\frac{2}{\a}+1\Big]\;\;\;\;\;\;\text{and}\;\;\;\;\;\;k=m+3.
$$


Next, we consider $A_1$ in \eqref{gdnls-10}. Thus
\begin{equation}\label{gdnls-14}
\begin{split}
A_1 &\equiv \|\p_x(|u|^{\a})\p_x^ku\|_{L^1_xL^2_T}\\
&\le c\,\|\ji x\jd \p_x(|u|^{\a})\p_x^ku\|_{L^2_xL^2_T}\\
&\le c\,\|\ji x\jd |u|^{\a-1}|\p_xu||\p_x^ku|\|_{L^2_TL^2_x}\\
&\le c\,T^{1/2}\|\ji x\jd \ji x\jd^{m(1-\a)} \p_xu\|_{L^{\infty}_TL^{\infty}_x}\|\p_x^{k}u\|_{L^{\infty}_TL^2_x}\\
&\le c\,T^{1/2}\|\ji x\jd^{m-1}\p_xu\|_{L^{\infty}_TL^{\infty}_x}\|\p_x^{k}u\|_{L^{\infty}_TL^2_x}
\end{split}
\end{equation}
by combining \eqref{gdnls-11} and the fact that $m\alpha\geq2$, with $\,c=c(\lambda;\alpha)$. Above we have used that
\begin{equation}\label{gdnls-15}
\begin{split}
\p_x(|u|^{\a})&=\p_x(u\overline{u})^{\frac{\a}{2}}=\frac{\a}{2} (u\overline{u})^{\frac{\a}{2}-1}(\overline{u}\p_xu+u\p_x\overline{u})\\
&=\frac{\a}{2} |u|^{\a-2}(\overline{u}\p_xu+u\p_x\overline{u}),
\end{split}
\end{equation}
and so
\begin{equation}\label{gdnls-16}
|\p_x |u|^{\a}|\le c\,|u|^{\a-1} |\p_xu|.
\end{equation}

We now turn to the estimate of $A_2$ in \eqref{gdnls-10}

\begin{equation}\label{gdnls-17}
\begin{split}
A_2 &\equiv \|\p_x^2(|u|^{\a}) \p_x^{k-1}u\|_{L^1_xL^2_T}\\
&\le  \|\ji x \jd \p_x^2(|u|^{\a}) \p_x^{k-1}u\|_{L^2_xL^2_T}\\
&\le \|\ji x\jd |u|^{\a-1}|\p_x^2u||\p_x^{k-1}u|\|_{L^2_xL^2_T}+\|\ji x\jd |u|^{\a-2}|\p_xu|^2|\p_x^{k-1}u|\|_{L^2_xL^2_T}\\
&\equiv A_{2,1}+A_{2,2}
\end{split}
\end{equation}
where from \eqref{gdnls-11}
\begin{equation}\label{gdnls-18a}
\begin{split}
A_{2,1} &\le cT^{1/2}\|\ji x\jd^{m(1-\a)+1}|\p_x^2u||\p_x^{k-1}u|\|_{L^{\infty}_TL^2_x}\\
&\le c T^{1/2}\|\ji x\jd^{m-1} \p_x^2 u\|_{L^{\infty}_TL^2_x} \|\p_x^{k-1}u\|_{L^{\infty}_TL^{\infty}_x}
\end{split}
\end{equation}
and 
\begin{equation}\label{gdnls-18}
\begin{split}
A_{2,2}&\leq c T^{1/2}\|\ji x\jd^{m(2-\a)+1}(\p_x u)^2\,\p_x^{k-1}u\|_{L^{\infty}_TL^{2}_x}\\
&\leq c T^{1/2}\|\ji x\jd^{m}\p_x u\|_{L^{\infty}_TL^{\infty}_x}^2 \|\p_x^{k-1}u\|_{L^{\infty}_TL^{2}_x}
\end{split}
\end{equation}

A familiar argument provides the estimates for $A_3, \dots, A_{k-1}$, so we shall consider $A_k$.

\begin{equation}\label{gdnls-19}
\begin{split}
A_k&\!\equiv \!\|\p_x^k(|u|^{\a})\p_xu\|_{L^1_xL^2_T}\\
&\le \!c\|\ji x\jd \p_x^k(|u|^{\a})\p_xu\|_{L^2_TL^2_x}\\
&\le \!c\|\ji x\jd |u|^{\a-k}|\p_xu|^k|\p_xu| \|_{L^2_TL^2_x} +\!\dots\\
&\hskip10pt + c\|\ji x\jd |u|^{\a-1}|\p_x^ku| |\p_xu|\|_{L^2_TL^2_x}\\
&\equiv \!A_{k,k}+\dots+A_{k,1}.
\end{split}
\end{equation}

Notice that $A_{k,1}$ was already estimated in \eqref{gdnls-14}. Roughly,  the middle terms $A_{k, k-1},\dots,A_{k,2}$ can be estimated
by interpolation between the result for $A_{k,k}$ and $A_{k,1}$, so we shall just consider $A_{k,k}$. Using \eqref{gdnls-11} and Sobolev embbeding

\begin{equation}\label{gdnls-20}
\begin{split}
A_{k,k} &\le c\|\ji x\jd \ji x\jd^{(k-\a)m} |\p_xu|^k\,|\p_xu|\|_{L^2_TL^2_x}\\
&\le c T^{1/2} \|\ji x\jd^{km-1} |\p_xu|^{k+1} \|_{L^{\infty}_TL^2_x}\\
&\le c T^{1/2} \| \ji x\jd^{m} \p_xu\|^k_{L^{\infty}_TL^{\infty}_x} \|\p_x u\|_{L^{\infty}_TL^2_x}\\
&\le c T^{1/2}\big(\underset{j=1}{\overset{3}{\sum}} \| \ji x\jd^{m} \p_x^ju\|_{L^{\infty}_TL^2_x}\big)^k \|\p_x u\|_{L^{\infty}_TL^2_x}\
\end{split}
\end{equation}
with $c=c(\lambda;\alpha)>0$.

Next, we bound $\|\Phi u(t)\|_{L^{\infty}_TL^2_x}$ as
\begin{equation}\label{gdnls-21}
\begin{split}
\|\Phi u(t)\|_{L^{\infty}_TL^2_x} &\le \|u_0\|_2 +\int\limits_0^t \| |u|^{\a}\p_xu\|_2(t')\,dt'\\
& \le \|u_0\|_2+ c T \||u|^{\a} \p_xu\|_{L^{\infty}_TL^2_x}\\
& \le \|u_0\|_2 +c T \| u\|_{L^{\infty}_xL^{\infty}_T}^{\alpha} \| \p_x u\|_{L^{\infty}_T L^2_x}.
\end{split}
\end{equation}
The Sobolev embedding yields the result since $k\ge 4$. Next, using the identity
\begin{equation}\label{gdnls-22}
x\,e^{it\p_x^2} u_0= e^{it\p_x^2} \,xu_0 +2i t \,e^{it\p_x^2}\p_x u_0=e^{it\partial_x^2}(x+2it\partial_x)u_0
\end{equation}
one has that
\begin{equation}\label{gdnls-22a}
x^me^{it\p_x^2} u_0=e^{it\partial_x^2}(x+2it\partial_x)^m u_0.
\end{equation}

We shall estimate 
\begin{equation}\label{gdnls-23}
 \sum_{j=1}^3 \| \ji x\jd^{m} \p_x^j\Phi(u)\|_{L^{\infty}_TL^{2}_x}\equiv \sum_{j=1}^3 E_j.
\end{equation}

Combining \eqref{gdnls-22a} and the interpolation inequalities in Section 2 it follows that
\begin{equation}
\label{a1}
\begin{split}
\| \ji x\jd^m\,e^{it\partial_x^2}u_0\|_2 &\leq c \| (1+|x|^m)\,e^{it\partial_x^2}u_0\|_2\\
&\leq c\| e^{it\partial_x^2}u_0\|_2 +c\|e^{it\partial_x^2}(x+2it\partial_x)^mu_0\|_2\\
&\leq c\|u_0\|_2 +c\|(x+2it\partial_x)^m u_0\|_2\\
&\leq c\big(\|\ji x\jd^m u_0\|_2+ \ji t\jd^m (\| u_0\|_2 +\|\partial_x^mu_0\|_2)\big).
\end{split}
\end{equation}




Thus, to estimate $E_1$, $E_2$ and $E_3$ in \eqref{gdnls-23} we write 
\begin{equation}
\label{a4}
\begin{aligned}
\| &\ji x\jd^m \,\partial_x^j e^{it\partial_x^2}u_0\|_2 =\|   \ji x\jd^m\,e^{it\partial_x^2}(\partial_x^ju_0)\|_2  \\
&\leq c\big(\|\,\ji x\jd^m \partial_x^j u_0\|_2+ \ji t\jd^m(\|\,\partial_x^ju_0\|_2 +\|\,\partial_x^{m+j}u_0\|_2)\big), \;j=1, 2, 3.
\end{aligned}
\end{equation}

For $E_3$ inserting \eqref{a4} into \eqref{gdnls-5} leads us to
\begin{equation}
\label{a5}
\begin{aligned}
&\| \ji x\jd^m\,\partial_x^3 \Phi(u)(t)\|_2\\
&\leq c\big(\| \ji x\jd^m\, \partial_x^3u_0\|_2+ \ji t\jd^m(\|\partial_x^3u_0\|_2 +\|\,\partial_x^{m+3}u_0\|_2)\big)\\
&\hskip10pt+c\int_0^t\|\ji x\jd^m \,\partial_x^3(|u|^{\alpha}\partial_xu)\|_2(t')dt'\\
&\hskip10pt+c\int_0^t \ji t-t'\jd^m(\|\,\partial_x^3(|u|^{\alpha}\partial_xu)\|_2+\|\partial_x^{m+3}(|u|^{\alpha}\partial_xu)\|_2)(t')dt'.
\end{aligned}\end{equation}

The new terms in the right hand side of \eqref{a5} are
\begin{equation}
\label{aa5}
E_{3,1}(t)\equiv \int_0^t\|\ji x\jd^m \,\partial_x^3(|u|^{\alpha}\partial_xu)\|_2(t')dt'
\end{equation}
and 
\begin{equation}
\label{aa6}
E_{3,2}(t)\equiv \int_0^t \ji t-t'\jd^m\|\partial_x^{m+3}(|u|^{\alpha}\partial_xu)\|_2(t')dt'.
\end{equation}

To estimate $E_{3,1}(t)$ we observe that
\begin{equation}
\label{6a}
\begin{aligned}
\sup_{t\in[0,T]} E_{3,1}(t)&\leq cT \|\ji x\jd^m \,\partial_x^3(|u|^{\alpha}\partial_xu)\|_{L^{\infty}_TL^2_x}\\
&\leq c T \,\sum_{j=0}^3\,\|\ji x\jd^m \,\partial_x^j(|u|^{\alpha})\partial^{4-j}_xu\|_{L^{\infty}_TL^2_x}\\
&\equiv c T \,\sum_{j=0}^3E_{3,1, j}.
\end{aligned}
\end{equation}

It will suffice to consider $E_{3,1,0} $ and $E_{3,1,3}$ in \eqref{6a}. Thus
\begin{equation}
\label{7a}
\begin{aligned}
E_{3,1,0} &\leq  c\,T \| \ji x\jd^{m}|u|\|_{L^{\infty}_TL^{\infty}_x}^{\alpha} \,\|\ji x\jd^{(m-m\alpha)}\,\partial_x^4u\|_{L^{\infty}_TL^2_x}\\
&\leq  c\,T \| \ji x\jd^m|u|\|_{L^{\infty}_TL^{\infty}_x}^{\alpha} \,\|\ji x\jd^{(m-2)}\,\partial_x^4u\|_{L^{\infty}_TL^2_x}.
\end{aligned}
\end{equation}

From the interpolation estimates in section 2 one sees that the last term in \eqref{7a} in fact
$$
\| \ji x\jd^j\,\partial_x^{m+2-j}u\|_2,\;\;\;\;\;\;j=0,1,...,m,
$$
can be interpolated between the cases : $j=0$ and $j=m$, i.e.
$$
\| \partial_x^2u\|_{m,2}\;\;\;\;\;\;\;\;\text{and}\;\;\;\;\;\;\;\| \ji x\jd^{m}\,\partial_x^2u\|_2,
$$
which provides the bound for the last term in \eqref{7a}. $E_{3,1,3}$ can be bounded as
\begin{equation}
\label{8a}
\begin{aligned}
E_{3,1,3} &\leq  c\, \| \ji x\jd^m\,\partial_x^3(|u|^{\alpha})\,\partial_xu\|_{L^{\infty}_TL^2_x}\\
&\leq  c\, \| \ji x\jd^{m}\,|u|^{\alpha-3}(\partial_xu)^3\partial_xu\|_{L^{\infty}_TL^2_x}\\
&\hskip10pt+c\, \| \ji x\jd^{m}\,|u|^{\alpha-2}\partial_x^2u\,\partial_xu\, \partial_xu\|_{L^{\infty}_TL^2_x}\\
&\hskip10pt+c\,\| \ji x\jd^{m}\,|u|^{\alpha-1}\partial_x^3u\,\partial_xu\|_{L^{\infty}_TL^2_x}\\
&\equiv c\,\sum_{j=1}^3\,E_{3,1,3,j}.
\end{aligned}
\end{equation}
Thus, from \eqref{gdnls-11} it follows that
$$
\aligned
E_{3,1,3,1}&\leq c\| \ji x\jd^{m+(3-\alpha)m}\,|\partial_xu|^4\|_{L^{\infty}_TL^2_x}\\
&\leq c\|\ji x\jd^{4m-1} |\partial_xu|^4\|_{L^{\infty}_TL^{\infty}_x}\\
&\leq c\| \ji x\jd^{m}\,\partial_xu\|_{L^{\infty}_TL^{\infty}_x}^4,
\endaligned
$$

$$
\aligned
E_{3,1,3,2}&\leq c\| \ji x\jd^{m+(2-\alpha)m}\,|\partial_xu|^2\partial_x^2u\|_{L^{\infty}_TL^2_x}\\
&\leq c\| \ji x\jd^{m}\,\partial_xu\|_{L^{\infty}_TL^{\infty}_x}^2\| \ji x\jd^{m}\,\partial_x^2u\|_{L^{\infty}_TL^{2}_x},
\endaligned
$$

and

$$
\aligned
E_{3,1,2,3}&\leq c\| \ji x\jd^{m+(1-\alpha)m}\,\partial_x^3u\,\partial_xu\|_{L^{\infty}_TL^{2}_x}\\
&\leq c\| \ji x\jd^{2m-2}\,\partial_x^3u\,\partial_xu\|_{L^{\infty}_TL^{2}_x}\\
&\leq c\| \ji x\jd^{m}\,\partial_xu\|_{L^{\infty}_TL^{\infty}_x}
\|\ji x\jd^{m-2}\,\partial_x^3u\|_{L^{\infty}_TL^{2}_x},
\endaligned
$$
which can be controlled by Sobolev embedding.

It remains to bound $\,E_{3,2}$ in \eqref{aa6}. Thus,
\begin{equation}
\label{a10}
\begin{aligned}
\sup_{t\in[0,T]}E_{3,2}(t)&= \sup_{t\in[0,T]}\int_0^t\ji t-t'\jd^m\|\partial_x^{m+3}(|u|^{\alpha}\partial_xu)\|_2(t')\,dt'\\
&\leq c T^{1/2} (1+T^m)^{1/2}\,\|\partial_x^{m+3}(|u|^{\alpha}\partial_xu)\|_{L^{2}_TL^2_x}\\&\leq cT^{1/2}\,\sum_{j=0}^{m+3}\|\partial_x^j(|u|^{\alpha})\partial_x^{m+4-j}u\|_{L^{2}_TL^2_x}\\
&\equiv cT ^{1/2}\sum_{j=0}^{m+3}E_{3,2,j}.
\end{aligned}
\end{equation}

As before it will suffice to consider the terms $E_{3,2,0}$ and $E_{3,2,m+3}$ in \eqref{a10}.

Thus, since $s=k+1/2=m+3+1/2$ one has that $s+1/2=m+4$ and 
\begin{equation}
\label{a11}
\begin{aligned}
E_{3,2,0}&=\||u|^{\alpha}\partial_x^{s+1/2}u\|_{L^{2}_xL^2_T}\\
&\leq c\| \,|u|^{\alpha}\|_{L^2_xL^{\infty}_T} \|\partial_x^{s+1/2}u\|_{L^{\infty}_xL^2_T}\\
&\leq c\| \,\ji x\jd\,|u|^{\alpha}\|_{L^{\infty}_TL^{\infty}_x} \|\partial_x^{s+1/2}u\|_{L^{\infty}_xL^2_T}\\
&\leq c\| \,\ji x\jd^{1/\alpha}\,|u|\|_{L^{\infty}_TL^{\infty}_x} ^{\alpha}\|\partial_x^{s+1/2}u\|_{L^{\infty}_xL^2_T}\\
&\leq c\| \,\ji x\jd^{m}\,|u|\|_{L^{\infty}_TL^{\infty}_x} ^{\alpha}\|\partial_x^{s+1/2}u\|_{L^{\infty}_xL^2_T}.
\end{aligned}
\end{equation}
Using  the fact that $\,1/\alpha<m\,$ one gets the appropriate bound for the term $E_{3,2,0}$ in \eqref{a11}.

Finally, we recall that $k=m+3$ thus the bound of the term $E_{3,2,m+3}$, which involves a $L^2_xL^2_T$-norm is similar, in fact simpler, that that carried out in \eqref{gdnls-19}
for the term $A_k$ involving an estimated in  the $\,L^1_xL^2_T$-norm. This completes the estimate of the term $E_3$ in \eqref{gdnls-23}. The proof for the term $E_1$ in \eqref{gdnls-23} is similar, in fact, simpler, hence it will be omitted.

Now we shall estimate 
\begin{equation}\label{gdnls-123}
  \| \ji x\jd^{m} \Phi(u)\|_{L^{\infty}_TL^{\infty}_x}\equiv  E_0.
\end{equation}
and  show that
\begin{equation*}
\underset{(x,t)\in\R\times[0,T]}{\rm Inf}  \ji x\jd^m  |\Phi(u)(x,t)|\ge \frac{\lambda}{4}
\end{equation*}
for sufficiently small $T>0$.

We recall that
\begin{equation}\label{gdnls-30}
\Phi(u)(t)=e^{it\p_x^2}u_0+\mu \int\limits_0^t e^{i(t-t')\p_x^2} (|u|^{\a}\p_xu)(t')\,dt'.
\end{equation}

First, we observe that for $t\in[0,T]$ the Sobolev embedding yields the estimate
\begin{equation}\label{gdnls-31}
\|e^{it\p_x^2}u_0-u_0\|_{\infty}=\| \int\limits_0^t \frac{d}{d\tau} \Big(e^{i\tau\p_x^2}u_0\Big)\,d\tau\|_{\infty}=\|\int_0^te^{i\tau\p_x^2}\partial_x^2u_0\,d\tau\|_{\infty}.
\end{equation}

Thus, combining \eqref{gdnls-31} and \eqref{a1} it follows that
\begin{equation}
\label{q1}
\begin{aligned}
\| \ji x\jd^m e^{it\p_x^2}u_0\|_{\infty}&\leq\| \ji x\jd^m (e^{it\p_x^2}u_0-u_0)\|_{\infty} + \|\ji x\jd^m u_0\|_{\infty}\\
&\leq c\| \ji x\jd^m \int_0^t\frac{d\;}{d\tau}e^{i\tau\p_x^2}u_0\,d\tau\|_{\infty}+c \|\ji x\jd^m u_0\|_{\infty}\\
&\leq c \| \ji x\jd^m \int_0^t\,e^{i\tau\p_x^2}\partial_x^2u_0\,d\tau\|_{\infty}+ c\|\ji x\jd^m u_0\|_{\infty}\\
&\leq ct\sup_{\tau\in[0,t]}\| \ji x\jd^m e^{i\tau\p_x^2}\partial_x^2u_0\|_{1,2}+ c\|\ji x\jd^m u_0\|_{\infty}\\
&\leq c t (\|  \ji x\jd^m\partial_x^{2}u_0\|_{1,2}+ \ji t\jd^m\| \partial_x^2u_0\|_{m+1,2})\\
&\hskip10pt + c\|\ji x\jd^m u_0\|_{\infty} \\
&\leq c_T\,t\,\delta,
\end{aligned}
\end{equation}

and 
\begin{equation}\label{gdnls-32}
\begin{split}
\| \ji x\jd^m (e^{it\p_x^2}u_0-u_0)\|_{\infty}&=\| \int\limits_0^t \frac{d}{d\tau} \Big( \ji x\jd^m e^{it\p_x^2}u_0(x)\Big)\,d\tau\|_{\infty}\\
&\le t\Big(\| \ji x\jd^m \partial_x^2u_0\|_{1,2}+ c_T \|\p_x^{m+2}u_0\|_2\Big)\\
&\le c\,t\delta,
\end{split}
\end{equation}
since $s=k+1/2$ and $k\ge m+3$. Hence, for any $(x,t)\in \R\times[0,T]$
\begin{equation}\label{gdnls-33}
\begin{split}
\big |\ji x\jd^m e^{it\p_x^2}u_0(x)\big| &\ge |  \ji x\jd^m u_0(x)| - |\ji x\jd^m (e^{it\p_x^2}u_0-u_0)(x)|\\
\\
&\ge \lambda- c\,T\delta\geq \lambda/2,
\end{split}
\end{equation}
for $\,T$ sufficiently small. Finally, we need to estimate
\begin{equation}\label{gdnls-34}
\begin{split}
B_1&\equiv \| \ji x\jd^m\int\limits_0^t e^{i(t-t')\p_x^2} (|u|^{\a}\p_x u)(t')\,dt' \|_{\infty}\\
&\le c \|  \ji x\jd^m\int\limits_0^t e^{i(t-t')\p_x^2} (|u|^{\a}\p_x u)(t')\,dt'\|_{1,2}\\
&\le c T\big(\| \ji x\jd^m\p_x(|u|^{\a}\p_x u)\|_{L^{\infty}_TL^2_x}
+\| \ji x\jd^m  |u|^{\a}\p_x u \|_{L^{\infty}_TL^2_x}\\
&\quad+ c_T\| |u|^{\a}\p_x u\|_{L^{\infty}_TL^2_x}
+c_T\| \p_x^{m+1}( |u|^{\a}\p_x u)\|_{L^{\infty}_TL^2_x}\big).
\end{split}
\end{equation}

But since $m+3=k$ all these terms in the right hand side of \eqref{gdnls-34} have been already bounded. Hence using the factor $T$ in the front we have that
\begin{equation}\label{gdnls-35}
B_1\le c T ((c\delta)^{\a}\,c\delta +c_T (\delta+\delta^k) c\delta).
\end{equation}
Therefore, combining \eqref{gdnls-30}-\eqref{gdnls-34}  for $T$ sufficiently small depending on $\lambda$ and $\delta$, we deduce that
\begin{equation}\label{gdnls-36}
\begin{split}
\underset{(x,t)\in\R\times[0,T]}{\rm Inf} \ji x\jd^m |\Phi(u)(x,t)| &\ge \frac{\lambda}{2}- c T ((c\delta)^{\a}\,c\delta +c_T (\delta+\delta^k) c\delta)\\
&\ge \frac{\lambda}{4}.
\end{split}
\end{equation}

Gathering the above information and recalling the notation
\begin{equation}\label{gdnls-37}
\begin{aligned}
\tres v\tres_{_T} &= \|v(t)\|_{L^{\infty}_TH^s} +\|\p_x^{k+1} v\|_{L^{\infty}_xL^2_T}\\
& +\|\ji x\jd^m\,u\|_{L^{\infty}_TL^{\infty}_x}+\sum_{j=0}^3\|\ji x\jd^m \partial_x^{j+1}u\|_{L^{\infty}_TL^2_x}
 \end{aligned}
\end{equation}
one has 
\begin{equation}\label{gdnls-38}
\tres \Phi(u)\tres_{_T} \le c\delta + c\delta^{\a}(c \delta)+cT^{1/2}(1+T^{1/2})(\delta^{1+\alpha}+\delta^k)\le 2c\delta
\end{equation}
if 
\begin{equation}\label{gdnls-39}
 c\delta^{\a}+cT^{1/2}(1+T^{1/2})(\delta^{\alpha}+\delta^{k-1})\ll 1.
\end{equation}

It remains to show that the operator $\Phi(u)$ defined in \eqref{gdnls-5} is a contraction in the $\tres\cdot\tres_{_T}$-norm.
\vskip.1in

We observe as before that 
\begin{equation}\label{cont-1}
\begin{split}
\underset{[0,T]} \sup\,\|&\Phi(u)(t)-\Phi(v)(t)\|_{s,2}\\
&\le \| \Phi(u)-\Phi(v)\|_{L^{\infty}_TL^2_x}+ \| D^s( \Phi(u)-\Phi(v))\|_{L^{\infty}_TL^2_x}.
\end{split}
\end{equation}

Using the smoothing effects \eqref{gdnls-4} we have
\begin{equation}\label{cont-2}
\begin{split}
\| D^s( &\Phi(u)-\Phi(v))\|_{L^{\infty}_TL^2_x}+\|\p_x^{k+1}(\Phi(u)-\Phi(v))\|_{L^{\infty}_xL^2_T}\\
&\le  c\|\p_x^k(|u|^{\a}\p_xu-|v|^{\a}\p_xv)\|_{L^1_xL^2_T}\\
&\le   c\||u|^{\a}\p_x^{k+1}u-|v|^{\a}\p_x^{k+1}v\|_{L^1_xL^2_T}\\
&\hskip10pt+ \||u|^{\a-k}(\p_xu)^{k+1}-|v|^{\a-k}(\p_x v)^{k+1}\|_{L^1_xL^2_T}\\
&\hskip10pt+\underset{j=1}{\overset{k-1}{\sum}}c_j\|\p_x^j(|u|^{\a})\p_x^{k+1-j}u-\p_x^j(|v|^{\a})\p_x^{k+1-j}v\|_{L^1_xL^2_T}\\
&\hskip10pt + \underset{j=1}{\overset{k-1}{\sum}}\, \widetilde{A}_{j,1}
 \end{split}
\end{equation}
where the terms  $\widetilde{A}_{j,1}$ are the corresponding ones to $A_{j,1}$ in \eqref{gdnls-19} and which will be estimated using interpolation.


From \eqref{a1} and \eqref{a4} it follows that
\begin{equation}\label{cont-3}
\begin{split}
\underset{j=1}{\overset{3}{\sum}}\|\ji x\jd^{m}\p_x^j(\Phi(u)-&\Phi(v))\|_{L^{\infty}_TL^2_x}
\le \underset{j=1}{\overset{3}{\sum}}\;\underset{[0,T]}{\sup}\,\widetilde{E}_j 
\end{split}
\end{equation}
where
\begin{equation}
\begin{split}
\widetilde{E}_j &=\int\limits_0^t\|\ji x\jd^m \p_x^j(|u|^{\a}\p_xu-|v|^{\a}\p_xv)\|_2(t')\,dt'\\
&\hskip10pt+\int\limits_0^t \ji t-t'\jd^m \|\p_x^j(|u|^{\a}\p_xu-|v|^{\a}\p_xv)\|_2 (t')\,dt'\\
&\hskip10pt+\int\limits_0^t \ji t-t'\jd^m \|\p_x^{m+j}(|u|^{\a}\p_xu-|v|^{\a}\p_xv)\|_2 (t')\,dt'.
\end{split}
\end{equation}

As in \eqref{gdnls-23} we will only consider the most difficult terms to bound $\widetilde{E}_3$ above. More precisely,
\begin{equation}
\widetilde{E}_{3,1,0}=\|\ji x\jd^m(|u|^{\a}\p_x^4u-|v|^{\a}\p_x^4v)\|_{L^{\infty}_TL^2_x}
\end{equation}
and
\begin{equation}
\widetilde{E}_{3,2,0}=\|\ji x\jd^m(|u|^{\a}\p_x^{m+4}u-|v|^{\a}\p_x^{m+4}v)\|_{L^2_TL^2_x}.
\end{equation}

To estimate $\|\ji x\jd^m (\Phi(u)-\Phi(v))\|_{L^{\infty}_TL^{\infty}_x}$ we use the Sobolev embedding and  \eqref{a1} to obtain
\begin{equation}\label{infty-1}
\begin{split}
\|\ji x\jd^m &(\Phi(u)-\Phi(v))\|_{L^{\infty}_TL^{\infty}_x} \\
&\le c\|\ji x\jd^m (|u|^{\a}\p_xu-|v|^{\a}\p_xv)\|_{L^1_TL^2_x}\\
&\hskip10pt+ c\|\ji x\jd^m \p_x (|u|^{\a}\p_xu-|v|^{\a}\p_xv)\|_{L^1_TL^2_x}\\
&\hskip10pt+ cT(1+T^m) \||u|^{\a}\p_xu-|v|^{\a}\p_xv\|_{L^{\infty}_TL^2_x}\\
&\hskip10pt+ cT^{1/2} (1+T^m) \|\p_x^{m+1} ( |u|^{\a}\p_xu-|v|^{\a}\p_xv)\|_{L^2_TL^2_x}.
\end{split}
\end{equation}

Hence, with the exception of $ \||u|^{\a}\p_xu-|v|^{\a}\p_xv\|_{L^{\infty}_TL^2_x}$  in \eqref{infty-1}, to estimate \eqref{cont-2}, \eqref{cont-3} and \eqref{infty-1}, we are reduce to consider the following terms
\begin{equation}\label{gdnls-40}
F_1= \| |u|^{\a}\p_x^{k+1}u -|v|^{\a}\p_x^{k+1}v\|_{L^1_xL^2_T},
\end{equation}
\begin{equation}\label{gdnls-41}
F_2= \| |u|^{\a-k}(\p_x u)^{k+1} -|v|^{\a-k}(\p_x v)^{k+1}\|_{L^1_xL^2_T},
\end{equation}
and
\begin{equation}\label{gdnls-42}
F_3= \| \ji x\jd^m \big(|u|^{\a}\p_x u -|v|^{\a}\p_x v\big)\|_{L^1_TL^2_x}.
\end{equation}

We shall use that if $x,y>0$, then there exist  $\widehat{\theta}\in(0,1) $ such that
\begin{equation*}
\big| x^{\beta} -y^{\beta}\big|= \beta |\widehat{\theta} x-(1- \widehat{\theta})y|^{\beta-1}|x-y|, \hskip15pt \beta\in\R.
\end{equation*}

In particular,  if $z, w\in \C$,
\begin{equation}\label{gdnls-43}
\begin{split}
\big| |z|^{\beta}- |w|^{\beta} \big| &= \beta\big|\,\widehat{\theta} |z|-(1- \widehat{\theta}\,)|w|\big|^{\beta-1}\big| |z|-|w|\big|\\
&\le \beta\,\big|\,\widehat{\theta}\, |z|+ (1- \widehat{\theta}\,)\,|w|\big|^{\beta-1}\, |z-w|.
\end{split}
\end{equation}

Then
\begin{equation}\label{gdnls-44}
\begin{split}
F_1 &\le \| |u|^{\a}\p_x^{k+1}(u-v)\|_{L^1_xL^2_T}+\| (|u|^{\a}-|v|^{\a})\p_x^{k+1}v\|_{L^1_xL^2_T}\\
&= F_{1,1}+F_{1,2}.
\end{split}
\end{equation}

The estimate for $F_{1,1}$ is similar to that in \eqref{gdnls-12}.

Now 
\begin{equation}\label{gdnls-45}
\begin{split}
&F_{1,2} \le \| \ji x\jd^2 (|u|^{\a}-|v|^{\a})\|_{L^{\infty}_xL^{\infty}_T}\| \p_x^{k+1}v\|_{L^{\infty}_xL^2_T}\\
&\le c_{\a} \,\| \ji x\jd^2 \big|\,\theta_{x,t}|u|+(1-\theta_{x,t})|v|\big|^{\a-1}|u-v| \|_{L^{\infty}_xL^{\infty}_T}\| \p_x^{k+1}v\|_{L^{\infty}_xL^2_T}\\
&= \widetilde{F}_{1,2}
\end{split}
\end{equation}

Since 
\begin{equation}\label{gdnls-46}
\begin{split}
\ji x\jd^m\, |u(x,t)| &\ge \lambda/4\\
\ji x\jd^m \,|v(x,t)| &\ge \lambda/4
\end{split}
\end{equation}
for any $(x,t)\in\R\times[0,T]$, using \eqref{gdnls-43} it follows that
\begin{equation}\label{gdnls-47}
\begin{split}
\widetilde{F}_{1,2} &\le  c_{a}\|  \ji x\jd^2 \ji x\jd^{m(1-\a)}|u-v|\|_{L^{\infty}_xL^{\infty}_T}\| \p_x^{k+1}v\|_{L^{\infty}_xL^2_T}\\
&\le \,c_{\a,\lambda} \|  \ji x\jd^m |u-v|\|_{L^{\infty}_xL^{\infty}_T}\| \p_x^{k+1}v\|_{L^{\infty}_xL^2_T}
\end{split}
\end{equation}
which yields the desired estimate for $F_1$.

Next consider $F_2$ in \eqref{gdnls-41}
\begin{equation}\label{gdnls-48}
\begin{split}
F_2 &\le  \| |u|^{\a-k}\big((\p_x u)^{k+1} -(\p_x v)^{k+1}\big)\|_{L^1_xL^2_T}\\
&\hskip10pt +\| \big(|u|^{\a-k} -|v|^{\a-k}\big)(\p_x v)^{k+1}\|_{L^1_xL^2_T}\\
&= F_{2,1}+ F_{2,2}.
\end{split}
\end{equation}

where $F_{2,1}$ was basically estimated in \eqref{gdnls-20}
So we consider $F_{2,2}$. From \eqref{gdnls-43}-\eqref{gdnls-46}
\begin{equation}\label{gdnls-49}
\begin{split}
&F_{2,2} =\| \big(|u|^{\a-k} -|v|^{\a-k}\big)(\p_x v)^{k+1}\|_{L^1_xL^2_T}\\
&\le \| \ji x\jd \big(|u|^{\a-k} -|v|^{\a-k}\big)(\p_x v)^{k+1}\|_{L^2_TL^2_x}\\
&\le \,c_{\a,\lambda}\, \| \ji x\jd \ji x\jd ^{m(k+1-\a)} |u-v|(\p_x v)^{k+1}\|_{L^2_TL^2_x}\\
&\le \,c_{\a,\lambda}\, T^{1/2}\|  \ji x\jd ^{m}\p_x v \|_{L^{\infty}_TL^{\infty}_x}^k
\| \p_x v\|_{L^{\infty}_TL^2_x}
\| \ji x\jd ^{m} |u-v| \|_{L^{\infty}_TL^{\infty}_x}
\end{split}
\end{equation}

Finally, we consider $F_3$ in \eqref{gdnls-42},
\begin{equation}\label{gdnls-50}
\begin{split}
F_3&\le \| \ji x\jd^m |u|^{\a}\p_x (u-v)\|_{L^1_TL^2_x}+\| \ji x\jd^m (|u|^{\a}-|v|^{\a})\p_x v\|_{L^1_TL^2_x}\\
&= F_{3,1}+ F_{3,2}
\end{split}
\end{equation}
Notice that
\begin{equation}
 \| \ji x\jd^m |u|^{\a}\p_x (u-v)\|_{L^1_TL^2_x}\le  c\,T\|u\|^{\alpha}_{L^{\infty}_TL^{\infty}_x}\| \ji x\jd^m \p_x (u-v)\|_{L^{\infty}_TL^2_x}.
 \end{equation}
Thus using \eqref{gdnls-43} one finds that
\begin{equation}\label{gdnls-51}
\begin{split}
F_{3,2} &= \| \ji x\jd^m (|u|^{\a}-|v|^{\a})\p_x v\|_{L^1_xL^2_T}\\
& \le  \| \ji x\jd^m  \ji x\jd^{m(1-\a)}|u-v| \p_x v\|_{L^1_xL^2_T}\\
&\le \,c_{\a,\lambda}\, T \|\ji x\jd^m |u-v| \ji x\jd^{m(1-\a)}\p_x v\|_{L^{\infty}_TL^2_x}\\
&\le \,c_{\a,\lambda} T  \|\ji x\jd^m |u-v| \|_{L^{\infty}_TL^{\infty}_x} \|  \ji x\jd^m \p_x v\|_{L^{\infty}_TL^2_x}.
\end{split}
\end{equation}

To end the proof we need to estimate $\|(\Phi(u)-\Phi(v))(t)\|_{L^{\infty}_TL^2_x}$ in \eqref{cont-1} and the third term on the right hand side of \eqref{infty-1}. Since the argument is similar we just prove the estimate for the first one.  Thus using \eqref{gdnls-43} it follows that
\begin{equation}
\begin{split}
&\|(\Phi(u)-\Phi(v))(t)\|_{L^{\infty}_TL^2_x}\\
&\le c\,T\||u|^{\a}-|v|^{\a}\|_{L^{\infty}_TL^{\infty}_x}\|\p_xu\|_{L^{\infty}_TL^2_x}\!\!+c\,T \|v\|_{L^{\infty}_TL^{\infty}_x}^{\a}\|\p_x(u-v)\|_{L^{\infty}_TL^2_x}\\
&\le c\,T \| \,\big|\theta_{x,t}|u|+(1-\theta_{x,t})|v|\big|^{\a-1}|u-v|\,\|_{L^{\infty}_TL^{\infty}_x}\|\p_xu\|_{L^{\infty}_TL^2_x}\\
&\hskip10pt+c\,T \|v\|_{L^{\infty}_TL^{\infty}_x}^{\a}\|\p_x(u-v)\|_{L^{\infty}_TL^2_x}\\
&\le c\,T\|\ji x\jd^{m(1-\a)}|u-v|\|_{L^{\infty}_TL^{\infty}_x}\|\p_xu\|_{L^{\infty}_TL^2_x}\\
&\hskip10pt + c\,T \|v\|_{L^{\infty}_TL^{\infty}_x}^{\a}\|\p_x(u-v)\|_{L^{\infty}_TL^2_x}\\
&\le c\,T\|\ji x\jd^{m}|u-v|\|_{L^{\infty}_TL^{\infty}_x}\|\p_xu\|_{L^{\infty}_TL^2_x}\\
&\hskip10pt+ c\,T \|v\|_{L^{\infty}_TL^{\infty}_x}^{\a}\|\p_x(u-v)\|_{L^{\infty}_TL^2_x}.\\
\end{split}
\end{equation}

Using the notation \eqref{gdnls-37} and the estimates above we obtain
\begin{equation}\label{gdnls-53} 
\begin{split}
\tres \Phi(u)-\Phi (v) \tres \le  c_{T} (\delta^{\a}+\delta^k+\delta)\tres u-v\tres \le c \delta^{\a} \tres u-v \tres.
\end{split}
\end{equation}

From \eqref{gdnls-39},  $c \delta^{\a}\ll 1$, then it follows that $\Phi$ defines a contraction.  Thus we can conclude that there exists a unique fixed point solving our problem. 

Since the remainder of the proof is  standard we will omit it.

\end{document}